\documentclass[12pt]{amsart}
\usepackage[cp850,latin1]{inputenc}
\usepackage{graphics}

\newtheorem{theorem}{Theorem}
\newtheorem{lemma}[theorem]{Lemma}

\theoremstyle{definition}

\theoremstyle{remark}

\usepackage{amssymb,amsmath}

\def\det{\mathop\mathrm{det}\nolimits}

\def\H{\mathbb{H}}

\def\n{\nabla}

\newcommand{\s}{\mbox{$\Sigma$}}

\newcommand{\R}{\mbox{${\mathbb R}$}}

\newcommand{\N}{\mbox{$\mathbb{H}^2\times\R_1$}}

\newcommand{\fle}{\mbox{$\rightarrow$}}

\newcommand{\rf}[1]{\mbox{(\ref{#1})}}
\newcommand{\rl}[1]{{~\ref{#1}}}

\newcommand{\fs}{\mbox{$\mathcal{C}^\infty(\s)$}}

\def\beq{\begin{equation}}
\def\eeq{\end{equation}}

\begin{document}

\title[A Hilbert theorem for spacelike surfaces in $\mathbb{H}^2\times\R_1$]
{A Hilbert-type theorem for spacelike surfaces with constant Gaussian
curvature in $\mathbb{H}^2\times\R_1$}

\author{Alma L. Albujer}
\address{Departamento de Estadística e Investigación Operativa, Universidad de Alicante, 03080 Alicante, Spain}
\email{albujer@um.es}

\author{Luis J. Al\'\i as}
\address{Departamento de Matem\'{a}ticas, Universidad de Murcia, E-30100 Espinardo, Murcia, Spain}
\email{ljalias@um.es}

\thanks{The authors are partially supported by MEC project MTM2007-64504, and Fundaci\'on S\'eneca project 04540/GERM/06, Spain. This
research is a result of the activity developed within the framework of the Programme in Support of Excellence Groups of the Regi\'on de Murcia, Spain, by Fundaci\'on S\'eneca, Regional Agency for Science and Technology (Regional Plan
for Science and Technology 2007-2010).}


\date{May 2009}



\begin{abstract}
There are examples of complete spacelike surfaces in the Lorentzian product $\H^2\times\R_1$ with constant Gaussian curvature $K\leq -1$. In this paper, we show that there exists no complete spacelike surface
in $\H^2\times\R_1$ with constant Gaussian curvature $K>-1$.
\end{abstract}

\maketitle

\section{Introduction}

In 1900 Liebmann \cite{Lie} characterized the spheres as the unique complete surfaces with constant positive Gaussian curvature in $\mathbb{R}^3$. One year later, in 1901 Hilbert \cite{Hil} showed that it does not exist any complete
surface with constant negative Gaussian curvature in $\mathbb{R}^3$. Finally, every complete surface with zero Gaussian curvature in $\mathbb{R}^3$ must be a straight cylinder over a complete, planar and simple curve, as was proved
independently by Hartman and Nirenberg in 1958 \cite{HaNi}, Stoker in 1961 \cite{Sto} and Massey in 1962 \cite{Mas}. The Liebmann and Hilbert theorems are easily extended to complete surfaces in $\mathbb{S}^3$ and $\mathbb{H}^3$,
since their proofs depend basically on the Codazzi equation, which is the same in any space form. In 2007 Aledo, Espinar and Gálvez \cite{AEG} extended the Liebmann and Hilbert theorems to the case of complete surfaces with constant
Gaussian curvature in the Riemannian homogeneous product spaces $\mathbb{S}^2\times\R$ and $\mathbb{H}^2\times\R$. Specifically, they showed that the only complete surfaces with constant Gaussian curvature $K >0$ in
$\mathbb{H}^2\times\R$ (resp. $K > 1$ in $\mathbb{S}^2\times\R$) are rotational surfaces. In addition, they proved the non existence of complete surfaces with constant Gaussian curvature $K < -1$ in $\mathbb{H}^2\times\R$ and
$\mathbb{S}^2\times\R$.

Recently, in \cite{AAA} the authors jointly with Aledo complemented the results in \cite{AEG} by showing that the slices are the only compact two-sided surfaces in $\mathbb{S}^2\times\R$ whose angle function does not change sign and
have constant Gaussian curvature. Moreover, a similar result is valid for spacelike complete surfaces in the Lorentzian product space $\mathbb{S}^2\times\R_1$: the only complete spacelike surfaces in the Lorentzian product
$\mathbb{S}^2\times\R_1$ with constant Gaussian curvature are the slices \cite[Corollary 9]{AAA}. However, in the proof of these results we use as a main tool the compactness of $\mathbb{S}^2$, so it can not be extended to surfaces in
the Lorentzian product $\mathbb{H}^2\times\R_1$. Actually, slices $\H^2\times\{t_0\},\,t_0\in\R$, are trivial examples of complete spacelike surfaces in $\mathbb{H}^2\times\R_1$ with constant Gaussian curvature $K=-1$. On the other
hand, in \cite[Example 12]{AAA} we have recently given an example of a non trivial complete entire spacelike graph in $\mathbb{H}^2\times\R_1$ with constant Gaussian curvature $K$ for every value of $K$
such that $K<-1$. Therefore, it seems a natural question to study the existence or non existence of complete spacelike
surfaces in $\H^2\times\R_1$ with constant Gaussian curvature $K >-1$. In this context, the following
non existence result is proved:

\begin{theorem}\label{thmain}
There exists no complete spacelike surface in $\mathbb{H}^2\times\R_1$ with constant Gaussian curvature $K>-1$.
\end{theorem}

The proof of Theorem\rl{thmain} for $K>0$ is a consequence of the Bonnet-Myers theorem taking into account that there is no compact surface in $\mathbb{H}^2\times\R_1$ (see Section\rl{proof}). On the other hand, in the case $-1<K\leq
0$ the proof follows the ideas introduced in \cite[Theorem 3]{AEG} and it is based on two geometric tools: the abstract theory of Codazzi pairs and the construction of a new complete metric on the surface obtained when we deform the
induced metric in the direction of the height function. However, in difference with the proof of \cite[Theorem 3]{AEG}, our proof of Theorem\rl{thmain} only requires tensorial computations.

In Section\rl{prel} we introduce the necessary notions about spacelike surfaces in $\mathbb{H}^2\times\R_1$ as well as the notion of a Codazzi pair and the theorem of Wissler, which is fundamental in the proof of Theorem\rl{thmain}. The
complete proof of Theorem\rl{thmain} is given in Section\rl{proof}. Finally, in the Appendix we compare the geometry of a spacelike surface in $\mathbb{H}^2\times\R_1$ with the geometry of the same surface endowed with the Riemannian
metric obtained by deformation of the induced metric by a fixed function.

\textit{Note added in proof.} After submission of this paper, we were informed by Gálvez, Jiménez and Mira that our
Theorem\rl{thmain} can be seen also as an application of their general correspondence results between isometric immersions
in \cite{GJM} and the non existence result of complete surfaces with constant Gaussian curvature $K<-1$ in the Riemannian
product $\mathbb{H}^2\times\R$.

\section{Preliminaries}
\label{prel}
\subsection{Spacelike surfaces in $\mathbb{H}^2\times\R_1$}\text\

Let $(\H^2,g_{\H^2})$ be the hyperbolic plane, and let us consider the product manifold $\H^2 \times \R$ endowed with the Lorentzian metric
\[
g=\pi_{\H}^*(g_{\H^2})-\pi_{\mathbb{R}}^*(dt^2),
\]
where $\pi_{\H}$ and $\pi_{\mathbb{R}}$ denote the projections from $\H^2 \times\R$ onto each factor. For simplicity, we will write
$$
g=g_{\H^2}-dt^2,
$$
and we will denote by \N\ the 3-dimensional product manifold $\H^2\times\R$ endowed with that Lorentzian metric.

A smooth immersion $f:\s\fle\N$ of a connected surface $\s^2$ is said to be a spacelike surface if $f$ induces a Riemannian metric on $\s$, which as usual is also denoted by $g$. It is interesting to remark that in that case, since
\[
\partial_t=(\partial / \partial_t)_{(x,t)}, \quad x \in \H^2, t \in \R,
\]
is a unitary timelike vector field globally defined on the ambient spacetime \N, there exists a unique unitary timelike normal field $N$ globally defined on \s\ which is in the same time-orientation as $\partial_t$. That is,
\[
g(N,\partial_t) \leq -1 < 0 \quad \mathrm{on} \quad \s.
\]
We will refer to $N$ as the future-pointing Gauss map of \s, and we will denote by $\Theta:\s \fle \left( -\infty,-1 \right]$ the smooth function on \s\ given by $\Theta=g(N,\partial_t)$. The function $\Theta$ measures the hyperbolic
angle $\theta$ between the future-pointing vector fields $N$ and $\partial_t$ along \s. Indeed, they are related by $\cosh \theta=-\Theta$.

In order to fix notation, let $\bar{\n}$ and $\n$ denote the Levi-Civita connections in \N\ and \s, respectively. Then the Gauss and Weingarten formulae for the spacelike surface $f:\s\fle\N$ are given by
\beq
\label{gaussfor}
\bar{\n}_XY=\n_XY-g(AX,Y)N
\eeq
and
\beq
\label{wein}
AX=-\bar{\n}_XN,
\eeq
for any tangent vector fields $X,Y \in T\s$. Here $A:T\s \fle T\s$ stands for the shape operator (or second fundamental form) of \s\ with respect to its
future-pointing Gauss map $N$. As is well known, the Gaussian curvature $K$ of the surface \s\ is described in terms of $A$ and the curvature of the ambient spacetime by the Gauss equation, which is given by
\beq
\label{eqgauss}
K=\bar{K}-\det A,
\eeq
where $\bar{K}$ denotes the sectional curvature in \N\ of the plane tangent to \s. It is not difficult to see that the Gauss equation \eqref{eqgauss} can be written as
\beq
\label{eqGauss}
K=-\Theta^2-\det A.
\eeq

On the other hand, let $\bar{R}$ denote the curvature tensor of \N. The Codazzi equation of the spacelike surface \s\ describes the tangent component of $\bar{R}(X,Y)N$, for any tangent vector fields $X,Y \in T\s$, in terms of the
derivative of the shape operator and it is given by
\beq
\label{codazzi}
(\bar{R}(X,Y)N)^\top=(\n_XA)Y-(\n_YA)X,
\eeq
where $\n_XA$ denotes the covariant derivative of $A$, that is,
\[
(\n_XA)Y=\n_X(AY)-A(\n_XY).
\]
From now on, if $Z$ is a vector field along the immersion $f:\s\fle\N$, then $Z^\top \in T\s$ stands for the tangential component of $Z$ along \s, that is, $Z=Z^\top-g(N,Z)N$. It can be seen that, as the hyperbolic plane is a complete
surface of constant Gaussian curvature $-1$, $\bar{R}$ can be simplified and the Codazzi equation \eqref{codazzi} becomes
\beq
\label{Codazzi}
(\n_XA)Y=(\n_YA)X -\Theta (g(X,\partial_t^\top)Y-g(Y,\partial_t^\top)X),
\eeq
(for the
details on the above computations see, for instance, \cite{A,AA}).

Given a spacelike surface $f:\s\fle\N$, the height function of \s, denoted by $h$, is defined as the projection of \s\ onto $\R$, that is, $h \in \fs$ is the smooth function given by $h=\pi_{\mathbb{R}} \circ f$. Observe that the
gradient of $\pi_\mathbb{R}$ on \N\ is
\[
\bar{\n}\pi_\mathbb{R}=-g(\bar{\n}\pi_\mathbb{R},\partial_t)\partial_t.
\]
Therefore, the gradient of $h$ on \s\ is
\[
\n h=(\bar{\n}\pi_\mathbb{R})^\top=-\partial_t^\top.
\]
Since $\partial_t^\top=\partial_t+\Theta N$, we easily get
\beq
\label{normgradh}
\| \n h \|^2=\Theta^2-1,
\eeq
where $\| \cdot \|$ denotes the norm of a vector field on \s. Since $\partial_t$ is parallel on \N\ we have that
\beq
\label{aux}
\bar{\n}_X\partial_t=0
\eeq
for any tangent vector field $X \in T\s$. Writing $\partial_t=-\n h - \Theta N$ along the surface \s\ and using Gauss \eqref{gaussfor} and Weingarten \eqref{wein} formulae, we easily get from
\eqref{aux} that
\beq
\label{aux2}
\n_X \n h=\Theta AX
\eeq
for every $X \in T\s$.

\subsection{Codazzi pairs}\text\

An important geometrical tool for the proof of our result is the abstract theory of Codazzi pairs following \cite{Mil}. Let $(A,B)$ be a pair of real quadratic forms on a $2$-dimensional surface \s\ such that $A$ is a Riemannian
metric. Associated to this pair it is possible to define its extrinsic curvature in an abstract way as the quotient
\beq
\label{extcurv}
K(A,B)=\frac{\det B}{\det A}.
\eeq
On the other hand, since $A$ is a Riemannian metric, it has
associated a Levi-Civita connection $\n^A$, a Riemann curvature tensor $R_A$ defined, as usual, by
\[
R_A(X,Y)Z=\n^A_{[X,Y]}Z-[\n^A_X,\n^A_Y]Z
\]
for any $X,Y,Z \in T\s$ and the
corresponding Gaussian curvature
\beq
\label{codintcurv}
K_A=\frac{A(R_A(X,Y)X,Y)}{Q_A(X,Y)},
\eeq
being
$Q_A(X,Y)=A(X,X)A(Y,Y)-A(X,Y)^2$ for any $X,Y\in T\s$.

The pair $(A,B)$ is said to be a Codazzi pair if it satisfies the
Codazzi equation of a space form, that is,
\beq
\label{codazzipair}
(\n^A_XS)(Y)-(\n^A_YS)(X)=0
\eeq
for every
$X,Y\in T\s$,  $S:T\s\fle T\s$ being the endomorphism in $T\s$
$A$-metrically equivalent to $B$, that is
\[
B(X,Y)=A(SX,Y),
\]
and $(\n^A_XS)$ the covariant derivative of $S$,
\[
(\n^A_XS)(Y)=\n^A_X(SY)-S(\n^A_XY).
\]
The following result, due to Wissler, will be
fundamental in the proof of our result:
\begin{theorem}[\cite{Wei}, \cite{Wis}]
\label{thwissler} Let $(A,B)$ be a Codazzi pair with constant
negative extrinsic curvature $K(A,B)$. Then, if $A$ is complete
$\inf_\Sigma |K_A|=0$.
\end{theorem}

\section{Proof of Theorem\rl{thmain}}
\label{proof} Let us recall first that any complete spacelike surface $f:\s\rightarrow\mathbb{H}^2\times\R_1$ is necessarily diffeomorphic to $\mathbb{H}^2$. Actually, it is not difficult to see that $\Pi=\pi_M\circ
f:\s\rightarrow\mathbb{H}^2$ satisfies $\Pi^\ast(g_{\mathbb{H}^2})\geq g$. Therefore, $\Pi$ is a local diffeomorphism which increases the distance between the Riemannian surfaces $(\s,g)$ and $(\mathbb{H}^2,g_{\mathbb{H}^2})$. The
completeness of \s\ implies that $\Pi$ is a covering map \cite[Chapter VIII, Lemma 8.1]{KoNo}. Moreover, since $\mathbb{H}^2$ is simply connected, $\Pi$ is a global diffeomorphism. As a direct consequence of it, there exists no
compact spacelike surface in $\mathbb{H}^2\times\R_1$. On the other hand, from the Bonnet-Myers theorem any Riemannian surface with positive constant Gaussian curvature is necessarily compact. Consequently, there exists no complete
spacelike surface in $\mathbb{H}^2\times\R_1$ with positive constant Gaussian curvature.

Let us assume now that $f:\s\fle\mathbb{H}^2\times\R_1$ is a complete spacelike surface with constant Gaussian curvature $-1<K\leq0$, and let us consider the Riemannian metric on \s\ defined by
\beq
\label{gtilde}
\tilde{g}=g+c\,dh^2\geq g,
\eeq
where $c$ is the positive constant $c=\frac{1}{K+1}>0$. Since $g$ is a complete metric by assumption and $\tilde{g}\geq g$, $\tilde{g}$ is also a complete metric on \s.

Let $\alpha:T\s\times T\s \fle\mathbb{R}$ denote the second fundamental form of the surface $f:\s\fle\mathbb{H}^2\times\mathbb{R}_1$, that is, $\alpha(X,Y)=g(AX,Y)$.

\vspace*{0.2cm}

\noindent {\bf Claim.} \hspace*{-0.4cm} {\it We assert that $(\tilde{g},\alpha)$ is a Codazzi pair with constant negative extrinsic curvature
\[
K(\tilde{g},\alpha)=-(K+1)<0.
\]}

To prove this claim, observe first that the endomorphism
$\tilde{A}:T\s \fle T\s$ which is $\tilde{g}$-metrically
equivalent to $\alpha$ can be written in terms of $A$. In fact for
any $X,Y \in T\s$ it holds
\[
g(AX,Y)=\alpha(X,Y)=\tilde{g}(\tilde{A}X,Y),
\]
and from \rf{gtilde}
\[
g(AX,Y)=\tilde{g}(AX,Y)-cAX(h)Y(h).
\]
Therefore we get \beq \label{auxatildea} \tilde{A}X=AX-c g(AX,\n
h)\tilde{\n}h,\eeq for any $X\in T\s$. On the other hand, by the
definition of the gradient of a function, and by the expression
\rf{gtilde} for the metric $\tilde{g}$, it yields
\[
X(h)=\tilde{g}(\tilde{\n}h, X)=g(\n h,X)=\tilde{g}(\n h, X)-c\|\n h\|^2\tilde{g}(\tilde{\n}h,X),
\]
for any $X\in T\s$. Then,
\[
\tilde{\n} h=\frac{1}{1+c\|\n h\|^2}\n h,
\]
so \rf{auxatildea} becomes \beq \label{atilde}
\tilde{A}X=AX-\frac{c}{1+c\|\n h\|^2}g(AX,\n h)\n h. \eeq

It is also possible to express the Levi-Civita connection of the
metric $\tilde{g}$, $\tilde{\n}$, in terms of the differential
operators related to the metric $g$, obtaining the relation \beq
\label{lctilde} \tilde{\n}_XY=\n_XY+\frac{c}{1+c\|\n h\|^2}\n^2
h(X,Y) \n h \eeq for any $X,Y\in T\s$, $\n^2$ being the Hessian
operator of the surface $f:\s\fle\mathbb{H}^2\times\mathbb{R}_1$,
(see the Appendix for the details).

From \rf{lctilde} and \rf{atilde} we get with a straightforward
computation that \beq \label{auxcp2}
\begin{split}
(\tilde{\n}_Y\tilde{A})X=&\tilde{\n}_Y(\tilde{A}X)-\tilde{A}(\tilde{\n}_XY)\\
=&(\n_YA)X-\frac{c}{1+c\|\n h\|^2}g((\n_YA)X,\n h)\n h\\
&-\frac{c}{1+c\|\n h\|^2}g(AX,\n h)\n_Y\n h+T(X,Y),
\end{split}
\eeq where $T$ is the symmetric $(0,2)$ tensor on \s\ given by
\[\begin{split}
T(X,Y)=&\frac{c^2}{(1+c\|\n h\|^2)^2}\Theta\left(g(AY,\n h)g(AX,\n h)+g(AX,Y)g(A(\n h),\n h)\right)\n h\\
&-\frac{c}{1+c\|\n h\|^2}\n^2 h(X,Y)A(\n h).
\end{split}
\]
Using the Codazzi equation \rf{Codazzi}, we observe that
\[
g((\n_YA)X-(\n_XA)Y,\n h)=0.
\]
Therefore, using again the Codazzi equation \rf{Codazzi} and the
expression \rf{aux2}, we obtain from \rf{auxcp2} that \beq
\label{auxcp}
\begin{split}
(\tilde{\n}_Y\tilde{A})X-(\tilde{\n}_X\tilde{A})Y=&\Theta\left(g(Y,\n
h)X-g(X,\n h)Y\right)\\&-\Theta\frac{c}{1+c\|\n
h\|^2}\left(g(AX,\n h)AY-g(AY,\n h)AX\right).
\end{split} \eeq To check that the left hand side of \rf{auxcp} vanishes, it is enough
to proof that it vanishes when we consider as vector fields
$\{E_1,E_2\}$ a local $g$-orthonormal frame of $T\s$ which
diagonalizes the shape operator. It is worth pointing out that
such a frame does not always exist; problems can occur when the
multiplicity of the principal curvatures changes and also at the
points where the principal curvatures are not differentiable.
However, we can consider the open dense subset of \s, $\s'$,
consisting of points at which the number of distinct principal
curvatures is locally constant. Then, for every $p\in \s'$ there
exists a local $g$-orthonormal frame defined on a neighbourhood of
$p$ that diagonalizes $A$, that is, $\{E_1,E_2\}$ such that
$AE_1=\lambda_1E_1$ and $AE_2=\lambda_2E_2$ with each $\lambda_i$
smooth, see, for instance, \cite[Paragraph 16.10]{Bes}. We will
work on $\s'$, and the conclusion will be valid in all the surface
\s\ by a continuity argument. Considering these vector fields,
\rf{auxcp} becomes \[
(\tilde{\n}_{E_2}\tilde{A})E_1-(\tilde{\n}_{E_1}\tilde{A})E_2=\Theta\left(1+\lambda_1\lambda_2\frac{c}{1+c\|\n
h\|^2}\right)\left(g(E_2,\n h)E_1-g(E_1,\n h)E_2\right),\] which
vanishes, since using the Gauss equation \rf{eqGauss} and the
relation \rf{normgradh} we get
\[
\lambda_1\lambda_2\frac{c}{1+c\|\n
h\|^2}=-(K+\Theta^2)\frac{\frac{1}{K+1}}{1+\frac{1}{K+1}\|\n
h\|^2}=-\frac{K+\Theta^2}{K+1+\|\n h\|^2}=-1.
\]

It remains to compute the extrinsic curvature of the Codazzi pair
$(\tilde{g},\alpha)$. Let $\{E_1,E_2\}$ be a local $g$-orthonormal
frame of $T\s$, then
\[
\tilde{g}(E_i,E_i)=1+cg(E_i,\n h)^2 \quad \mathrm{and} \quad
\tilde{g}(E_1,E_2)=cg(E_1,\n h)g(E_2,\n h).
\]
Therefore, we have
\[
\det \tilde{g}=(1+cg(E_1,\n h)^2)(1+cg(E_2,\n h)^2)-c^2g(E_1,\n
h)^2g(E_2,\n h)^2=1+c\|\n h\|^2,
\]
so using the equations \rf{eqGauss} and \rf{normgradh}, the
extrinsic curvature of $(\tilde{g},\alpha)$ is given by
\[
K(\tilde{g},\alpha)=\frac{\det\alpha}{\det\tilde{g}}=\frac{\det
A}{1+c\|\n h\|^2}=\frac{-(K+1)(K+\Theta^2)}{K+1+\|\n
h\|^2}=-(K+1)<0.
\]
This finishes the proof of our Claim.

Consider now $\s''\subset\s$ the subset in \s\ where the height
function $h$ is non constant. $\s''$ is an open dense subset of
\s, since in other case it would exist an open subset
$\Omega\subset\s$ where $h|_\Omega$  is constant. Then, from
expressions \rf{normgradh} and \rf{aux2} $\Theta|_\Omega=-1$ and
$A|_\Omega=0$. Therefore, from the Gauss equation \rf{eqGauss} it
would be $K=-1$, which contradicts our assumption. By Lemma 3 in
the Appendix, the Gaussian curvature of the surface
$(\s,\tilde{g})$, $\tilde{K}$, can be written in terms of the
Gaussian curvature of the surface $(\s,g)$ as \beq \label{aux2k}
\tilde{K}=\frac{K(1+c\|\n h\|^2)+c\det \n^2 h}{(1+c\|\n h\|^2)^2}
\eeq in $\s''$. And by continuity \rf{aux2k} holds in \s. Observe
that from the expressions \rf{aux2} and \rf{aux} and from the
Gauss equation \rf{eqGauss} we get \beq \label{det} \det{\n^2
h}=\Theta^2\det{A}=-\Theta^2(K+\Theta^2)=-(1+\|\n h\|^2)(K+1+\|\n
h\|^2). \eeq Therefore, \rf{aux2k} becomes \beq
\label{tildekfinal} \tilde{K}=\frac{(1-c)K-c(1+\|\n
h\|^2)^2}{(1+c\|\n h\|^2)^2}. \eeq If we consider in
\rf{tildekfinal} $\tilde{K}$ as a function of $\|\n h\|^2$, then
$\tilde{K}$ is a monotonous decreasing function. Therefore,
evaluating it at $0$ and using that $c=1/(K+1)$ we have \beq
\label{inftildek} \inf \tilde{K}\leq \sup
\tilde{K}=(1-c)K-c=K-1<0. \eeq

Summing up, we have proven that $(\tilde{g},\alpha)$ is a Codazzi
pair with negative constant extrinsic curvature, $\tilde{g}$ being
a complete Riemannian metric with Gaussian curvature $\tilde{K}$
verifying \rf{inftildek}, which contradicts the theorem of
Wissler, Theorem\rl{thwissler}. Therefore, it can not exist any
complete spacelike surface $f:\s\fle\H^2\times\R_1$  with constant
Gaussian curvature $-1<K\leq 0$, as we were assuming, which
completes the proof of Theorem\rl{thmain}.

\section*{Appendix: Relating the geometry of $(\s,g)$ and $(\s,\tilde{g})$.}

Given a Riemannian surface $(\s,g)$, a non constant smooth
function $u\in\fs$ and a positive constant $c>0$, it makes sense
to consider the new Riemannian surface $(\s,\tilde{g})$, where
\beq \label{gtildeap} \tilde{g}=g+cdu^2\geq g. \eeq Therefore
$(\s,\tilde{g})$ is obtained by deformation of the metric $g$ in
the direction of the function $u$. Observe that in the particular
case where \s\ is a spacelike surface in $\mathbb{H}^2\times\R_1$
and $u$ is the height function of \s, the situation is the one
presented in Section\rl{proof}. Our aim in this appendix is to
obtain some relations between the geometry of $(\s,g)$ and
$(\s,\tilde{g})$, giving general versions of the expressions
\rf{lctilde} and \rf{tildekfinal}.

We begin by studying the relation between the Levi-Civita
connections of $(\s,\tilde{g})$, $\tilde{\n}$, and $(\s,g)$, $\n$.
Using the Koszul formula and the expression \rf{gtildeap} for
$\tilde{g}$ we have
\[
\begin{split}
2\tilde{g}(\tilde{\n}_XY,Z)=&X(\tilde{g}(Y,Z))+Y(\tilde{g}(Z,X))-Z(\tilde{g}(X,Y))\\
&-\tilde{g}(X,[Y,Z])+\tilde{g}(Y,[Z,X])+\tilde{g}(Z,[X,Y])\\
=&2g(\n_XY,Z)+c\left[X(Y(u)Z(u))+Y(Z(u)X(u))-Z(X(u)Y(u))\right.\\
&\hspace*{2.85cm}-X(u)(YZ-ZY)(u)+Y(u)(ZX-XZ)(u)\\
&\hspace*{2.85cm}\left.+Z(u)(XY-YX)(u)\right]\\
=&2g(\n_XY,Z)+2cX(Y(u))Z(u),
\end{split}
\] for any $X,Y,Z \in T\s$. On the other hand, from \rf{gtildeap} we get
\[
\tilde{g}(\tilde{\n}_XY,Z)=g(\tilde{\n}_XY,Z)+c\tilde{\n}_XY(u)Z(u),
\]
so we obtain \beq \label{lctildeaux}
\tilde{\n}_XY=\n_XY-c\left(\tilde{\n}_XY(u)-X(Y(u))\right)\n u
\eeq for any $X,Y \in T\s$. It follows from here that
\[
\tilde{\n}_XY(u)=\n_XY(u)-c\tilde{\n}_XY(u)\|\n u\|^2+cX(Y(u))\|\n
u\|^2.
\]
Therefore, we have \beq \label{lctildeaux2}
\tilde{\n}_XY(u)=\frac{1}{1+c\|\n u\|^2}\left(\n_XY(u)+c
X(Y(u))\|\n u\|^2\right). \eeq Finally, substituting
\rf{lctildeaux2} into \rf{lctildeaux} we get \[
\tilde{\n}_XY=\n_XY-\frac{c}{1+c\|\n
u\|^2}\left(\n_XY(u)-X(Y(u))\right)\n u\] for any $X,Y \in T\s$.
Or equivalently, \beq \label{lctildeap}
\tilde{\n}_XY=\n_XY+\frac{c}{1+c\|\n u\|^2}\n^2 u(X,Y) \n u, \eeq
$\n^2$ being the Hessian operator of the surface $(\s,g)$.

In the following lemma, we obtain the relation between the
Gaussian curvature $\tilde{K}$ of $(\s,\tilde{g})$ and the
Gaussian curvature $K$ of $(\s,g)$.
\begin{lemma} \label{lemmakktilde}
Let $(\s,g)$ be a Riemannian surface, $u\in\fs$ a non constant
smooth function and $c>0$ a positive constant. Then, the Gaussian
curvature $\tilde{K}$ of the Riemannian surface
$(\s,\tilde{g}=g+cdu^2)$ is given by
\[
\tilde{K}=\frac{K(1+c\|\n u\|^2)+c\det \n^2 u}{(1+c\|\n u\|^2)^2},
\]
where $K$, $\n$ and $\n^2$ denote the Gaussian curvature, the
gradient and the Hessian operator of $(\s,g)$, respectively.
\end{lemma}
\begin{proof}
Let $\{E_1,E_2\}$ be a local $g$-orthonormal frame on $T\s$ such
that $E_2\perp\n u$. Then, \beq \label{kk} K=g(R(E_1,E_2)E_1,E_2),
\eeq  and \beq \label{tildekk}
\tilde{K}=\frac{\tilde{g}(\tilde{R}(E_1,E_2)E_1,E_2)}{\tilde{Q}(E_1,E_2)},
\eeq where
$\tilde{Q}(E_1,E_2)=\tilde{g}(E_1,E_1)\tilde{g}(E_2,E_2)-\tilde{g}(E_1,E_2)^2=1+c\|\n
u\|^2$, and $R$ and $\tilde{R}$ stand for the Riemann curvature
tensors of $(\s,g)$ and $(\s,\tilde{g})$, respectively. Therefore
we need the relation between $\tilde{R}$ and $R$. Since
\[\tilde{R}(E_1,E_2)E_1=\tilde{\n}_{[E_1,E_2]}E_1-[\tilde{\n}_{E_1},\tilde{\n}_{E_2}]E_1,\]
we will study each term separately. From the expression
\rf{lctildeap}, we have \beq \label{k11}
\begin{split}
\tilde{\n}_{\tilde{\n}_{E_1}E_2}E_1=&\tilde{\n}_{\n_{E_1}E_2}E_1+\frac{c}{1+c\|\n u\|^2}\n^2u(E_1,E_2)\tilde{\n}_{\n u}E_1\\
=&\n_{\n_{E_1}E_2}E_1+\frac{c}{1+c\|\n u\|^2}\n^2u(E_1,E_2)\n_{\n
u}E_1+f_1\n u,
\end{split}
\eeq and \beq \label{k12}
\begin{split}
\tilde{\n}_{\tilde{\n}_{E_2}E_1}E_1=&\tilde{\n}_{\n_{E_2}E_1}E_1+\frac{c}{1+c\|\n u\|^2}\n^2u(E_1,E_2)\tilde{\n}_{\n u}E_1\\
=&\n_{\n_{E_2}E_1}E_1+\frac{c}{1+c\|\n u\|^2}\n^2u(E_1,E_2)\n_{\n
u}E_1+f_2\n u,
\end{split}
\eeq where $f_1,f_2\in \mathcal{C}^\infty(\s)$. Observe that, in
order to obtain $\tilde{K}$, we will have to compute the product
of the expressions above times $E_2$, which is by assumption
orthogonal to $\n u$. Therefore, all the terms that are
proportional to $\n u$ will vanish, and so we do not mind the
explicit expressions for $f_1$ and $f_2$. From \rf{k11} and
\rf{k12} we get \beq \label{k1} \tilde{\n}_{[E_1,E_2]}E_1 =
\n_{[E_1,E_2]}E_1+f_3\n u, \eeq being $f_3=f_1-f_2\in\fs$. On the
other hand,
\[
\begin{split}
\tilde{\n}_{E_1}\tilde{\n}_{E_2}E_1=&\tilde{\n}_{E_1}\n_{E_2}E_1+\frac{c}{1+c\|\n u\|^2}\n^2u(E_1,E_2)\tilde{\n}_{E_1}\n u+f_4\n u\\
=&\n_{E_1}\n_{E_2}E_1+\frac{c}{1+c\|\n
u\|^2}\n^2u(E_1,E_2)\n_{E_1}\n u+f_5\n u,
\end{split}
\]
and
\[
\begin{split}
\tilde{\n}_{E_2}\tilde{\n}_{E_1}E_1=&\tilde{\n}_{E_2}\n_{E_1}E_1+\frac{c}{1+c\|\n u\|^2}\n^2u(E_1,E_1)\tilde{\n}_{E_2}\n u+f_6\n u\\
=&\n_{E_2}\n_{E_1}E_1+\frac{c}{1+c\|\n
u\|^2}\n^2u(E_1,E_1)\n_{E_2}\n u+f_7\n u,
\end{split}
\]
where again $f_4,f_5,f_6,f_7\in\fs$. Therefore, \beq \label{k2}
\begin{split}
[\tilde{\n}_{E_1},\tilde{\n}_{E_2}]E_1=&[\n_{E_1},\n_{E_2}]E_1\\&+\frac{c}{1+c\|\n
u\|^2}\left(\n^2u(E_1,E_2)\n_{E_1}\n u-\n^2u(E_1,E_1)\n_{E_2}\n
u\right)+f_8\n u
\end{split}
\eeq being $f_8=f_5-f_7\in\fs$, which jointly with \rf{k1} yields
\[\begin{split}
\tilde{R}(E_1,E_2)E_1=&R(E_1,E_2)E_1\\&+\frac{c}{1+c\|\n
u\|^2}\left(\n^2u(E_1,E_1)\n_{E_2}\n u-\n^2u(E_1,E_2)\n_{E_1}\n
u\right)+f\n u
\end{split}
\]
being $f=f_3-f_8\in\fs$. Therefore, \beq \label{auxk}
\begin{split}
\tilde{g}(\tilde{R}(E_1,E_2)E_1,E_2)=&g(\tilde{R}(E_1,E_2)E_1,E_2)\\=&g(R(E_1,E_2)E_1,E_2)+\frac{c}{1+c\|\n
u\|^2}\det{\n^2 u}. \end{split} \eeq Or equivalently, from \rf{kk} and \rf{tildekk}
\[
\tilde{K}=\frac{K(1+c\|\n u\|^2)+c\det \n^2 u}{(1+c\|\n u\|^2)^2},
\]
which proves the result.
\end{proof}

\bibliographystyle{amsplain}

\begin{thebibliography}{10}

\bibitem{A} A. L. Albujer, Geometría global de superficies espaciales en espacios producto lorentzianos, Ph.D. Thesis, Universidad de Murcia, Spain, 2008. Available at
http://www.tesisenred.net/TESIS$_{-}$UM/AVAILABLE/TDR-0204109-132118//AlbujerBrotons.pdf.

\bibitem{AAA} A. L. Albujer, J. A. Aledo and L. J. Alías, On the
scalar curvature of hypersurfaces in spaces with a Killing vector
field, to appear in Advances in Geometry. Available at http://arxiv.org/pdf/0906.2111.

\bibitem{AA} A. L. Albujer and L. J. Alías, Calabi-Bernstein results for maximal surfaces in Lorentzian
product spaces, {\it J. Geom. Phys.} {\bf 59} (2009), 620--631.

\bibitem{AEG} J. A. Aledo, J. M. Espinar and J. A. Gálvez,
Complete surfaces of constant curvature in $\H^2\times \R$ and
$\mathbb{S}^2\times \R$, {\it Calc. Variations \& PDEs} {\bf 29}
(2007), 347--363.

\bibitem{Bes} A. L. Besse, Einstein manifolds, Springer-Verlag, Berlin, 1987.

\bibitem{GJM} J.A. Gálvez, A. Jiménez and P. Mira, A correspondence for isometric immersions into product spaces and
applications, work in preparation.

\bibitem{HaNi} P. Hartman and L. Nirenberg, On spherical images whose jacobians
do not change signs, {\it Amer. J. Math.}, {\bf 81} (1959),
901--920.

\bibitem{Hil} D. Hilbert, Über Flächen von konstanter Gausscher Krümung,
{\it Trans. Am. Math. Soc.}, {\bf 2} (1901), 87--99.

\bibitem{KoNo} S. Kobayashi and K. Nomizu, {\it Foundations on Differential Geometry}, Vol. II, Interscience, New York, 1969.

\bibitem{Lie} H. Liebmann, Ueber die Verbiegung der geschlossenen Flächen
positiver Krümmung, {\it Math. Ann.} {\bf 53} (1900), 81--112.

\bibitem{Mas} W. S. Massey, Surfaces of Gaussian curvature zero in Euclidean
space, {\it Tohoku Math. J.} {\bf 14} (1962), 73--79.

\bibitem{Mil} T. K. Milnor, Abstract Weingarten Surfaces, {\it J. Diff. Geom.} {\bf
15} (1980), 365--380.

\bibitem{Sto} J. Stoker, Developable surfaces in the large, {\it Comm. Pure
Appl. Math.} {\bf 14} (1961), 627--635.

\bibitem{Wei} T. Weinstein, An introduction to Lorentz surfaces,
Walter de Gruiter, Berlin, New York, 1996.

\bibitem{Wis} C. Wissler, Globale Tschebyscheff-Netze auf Riemannschen
Mannigfaltigkeiten und Fortsetzung von Flächen konstanter
negativer Krümmung, {\it Comm. Math. Helv.} {\bf 47} (1972),
348--372.

\end{thebibliography}

\end{document}